\documentclass[12pt,fleqn,reqno,a4paper]{amsart}
\usepackage{enumerate,amsmath,amsthm,amssymb,cite,mathrsfs,mathtools}

\DeclareMathOperator{\real}{\mathrm{Re}}

\DeclareMathOperator{\SC}{\mathscr{S}^*_\mathcal{R}}
\numberwithin{equation}{section}
\newtheorem{theorem}{Theorem}[section]
\newtheorem{lemma}[theorem]{Lemma}
\newtheorem{corollary}[theorem]{Corollary}
\newtheorem{conjecture}[theorem]{Conjecture}
\theoremstyle{remark}
\newtheorem{remark}[theorem]{Remark}

\makeatletter
\@namedef{subjclassname@2010}{%
		\textup{2010} Mathematics Subject Classification}
\makeatother

\begin{document}

\title[Coefficients and Radius Estimates of   Starlike Functions]{Coefficients of  the Inverse Functions and Radius Estimates of  Certain  Starlike Functions}
\author[Adiba Naz]{Adiba Naz}

\address{Department of Mathematics, University of Delhi, Delhi--110 007, India}
\email{adibanaz81@gmail.com}

\author[Sushil Kumar]{Sushil Kumar}

\address{Bharati Vidyapeeth's College of Engineering, Delhi--110 063 India}
\email{sushilkumar16n@gmail.com}

\author[V. Ravichandran]{V. Ravichandran}

\address{Department of Mathematics, National Institute of Technology,  Tiruchirappalli--620 015,  India}
\email{vravi68@gmail.com}

\begin{abstract}Ma-Minda class (of starlike functions) consists of all normalized analytic functions $f$ on the unit disk for which the image of $zf'(z)/f(z)$ is contained in the some starlike region in the right-half plane. We obtain the best possible bounds on the second and third coefficient for the inverse functions of functions  in the Ma-Minda class.  The bounds on the Fekete-Szeg\"o functional and the second Hankel determinant of the inverse functions of the functions belonging to the Ma-Minda class   are also determined. Further, the bounds on the  first five coefficients of the  inverse functions are  investigated for two particular subclasses of  the  Ma-Minda class. In addition, some radius estimates associated with the two subclasses are also computed.
\end{abstract}

\keywords{Inverse coefficients, starlike functions, rational function, exponential function, Fekete-Szeg\"o inequality, Hankel determinants, radius problem}
\subjclass[2010]{30C45, 30C50, 30C80}

\maketitle

\section{Introduction}

Let $\mathscr{A}$ be a class of all functions $f$ that are analytic  in an open disk $\mathbb{D}:=\{z\in\mathbb{C}\colon|z|<1 \}$ and normalized by $f(0)=0$ and $f'(0)=1$. The subclass of $\mathscr{A}$ consisting of univalent functions is denoted by $\mathscr{S}$   and let class $\mathscr{P}$ denote the class of all functions  $p$ that are analytic  in $\mathbb{D}$ having positive real part and normalized by $p(0)=1$. An analytic function $f\colon\mathbb{D}\to\mathbb{C}$ is said to be subordinate to an analytic function $g\colon\mathbb{D}\to\mathbb{C}$, written as $f\prec g$, if there exists a Schwarz function $w$ with $w(0)=0$ and $|w(z)|<1$ that satisfies $f(z)=g(w(z))$ for $z\in\mathbb{D}$. In particular, if $g$ is univalent in $\mathbb{D}$, then $f\prec g$ if and only if $f(0)=g(0)$ and $f(\mathbb{D})\subseteq g(\mathbb{D})$.  Let a function $\varphi$ be  analytic and  univalent  with positive real part in $\mathbb{D}$ that maps $\mathbb{D}$ onto domains which are symmetric with respect to the real axis and starlike with respect to $\varphi(0)=1$ such that $\varphi'(0)>0$. For such a function $\varphi$, Ma and Minda \cite{MR1343506} gave a unified treatment  of various geometric properties such as growth, distortion, and covering theorems of the subclass    $\mathscr{S}^*(\varphi)$ defined by \begin{equation}\label{ma}\mathscr{S}^*(\varphi)=\left\{f\in\mathscr{A}\colon \frac{zf'(z)}{f(z)} \prec \varphi(z) \right\}.\end{equation} A class that unified several subclasses of starlike and convex functions were investigated by Shanmugam \cite{MR0994916} and he studied the convolution property of this and several related classes.   For $-1\leq B<A\leq 1$, $\mathscr{S}^*[A,B]:=\mathscr{S}^*((1+Az)/(1+Bz))$ is a well-known class consisting of Janowski starlike functions \cite{MR0328059}. The special case when $A=1-2\alpha$ and $B=-1$ with $0\leq \alpha< 1$ reduces to $\mathscr{S}^*(\alpha)$ consisting of starlike functions of order $\alpha$ \cite{MR783568}. In particular, $\mathscr{S}^*:=\mathscr{S}^*(0)$ is the class of starlike functions. Several  subclasses of starlike functions are special cases of $\mathscr{S}^*(\varphi)$ for varying superordinate function $\varphi$.
However this paper aims to consider the classes  $\mathscr{S}^*_e:=\mathscr{S}^*(e^z)$ consisting of functions $f\in\mathscr{A}$ such that $zf'(z)/f(z)$ lies in the domain $\{w\in\mathbb{C}\colon |\log w|<1\}$ and   $\SC:=\mathscr{S}^*(\varphi_\mathcal{R})$ where  \begin{equation}\label{rf}
\varphi_\mathcal{R}(z):= 1+\frac{z}{k}\left(\frac{k+z}{k-z}\right) = 1+\frac{1}{k}z+\frac{2}{k^2}z^2+\frac{2}{k^3}z^3+\cdots, \quad k=\sqrt{2}+1. \end{equation} The class $\mathscr{S}^*_e$ was  introduced by Mendiratta \textit{et al.\@} \cite{MR3394060} and the class $\SC$ was introduced by  Kumar and Ravichandran \cite{MR3496681}.

For a function $f\in\mathscr{S}$, there always exists an inverse function $f^{-1}$ defined on some disk $|\omega|<r_0(f)$ (where $r_0(f)\geq1/4$) having the Taylor series expansion
\begin{equation}
f^{-1}(\omega)=\omega+A_2 \omega^2+A_3 \omega^3+A_4 \omega^4+\cdots\label{inv}
\end{equation} near $\omega=0$.
In 1923, L\"owner \cite{MR1512136} investigated the sharp coefficient estimates for  the inverse function of $f\in\mathscr{S}$ using a parametric method. Later several authors \cite{MR3646796,MR681830,MR1140278,MR813267} started determining the initial coefficients of the inverse function belonging to  various subclasses of $\mathscr{S}$. Krzyz \textit{et al.\@} \cite{MR689590} obtained the sharp bounds  on first two coefficients of the inverse function  of a function lying in the class $\mathscr{S}^*(\alpha)$ and later their results were extended by Kapoor and Mishra \cite{MR2296897}.  Ali \cite{MR2055766} determined the first three coefficients of the inverse function  of strongly starlike functions of order $\alpha$  ($0<\alpha\leq1$). For $\beta>1$, Ali and Vasudevarao \cite{MR3436767} obtained the bounds of coefficients of the inverse function of functions belonging to  the class \[\mathscr{M}(\beta)=\left\{f\in\mathscr{A}\colon\real \dfrac{zf'(z)}{f(z)}<\beta,\; z\in\mathbb{D}  \right\}. \]
 This class was initially introduced by Uralegaddi \textit{et al.\@} \cite{MR1304483}.  Recently, Ravichandran and Verma \cite{MR3670865} discussed the inverse coefficient problem for functions in the class $\mathscr{S}^*[A,B]$ while Sok\'{o}{\l} and Thomas \cite{MR3796439} obtained the initial coefficients of the inverse function of functions belonging to  the class $\mathscr{S}^*_L:=\mathscr{S}^*(\sqrt{1+z})$ associated with lemniscate of Bernoulli which was introduced by Sok\'o{\l} and Stankiewicz \cite{MR1473947}.

The Hankel determinants play an important role in the study of the singularities and power series with integral coefficients.  The Hankel determinant $H_q(n)$ for a given function $f\in\mathscr{A}$ is the determinant of the matrix of order $q$ given by $H_q(n)=[a_{n+i+j-2}]$ where $a_1=1$ and $n$, $q$ are fixed positive integers.
 Determination of the exact bounds of $|H_q(n)|$ for various subclasses of analytic functions was investigated by many authors. Pommerenke \cite{MR0185105,MR0215976} first studied the Hankel determinant for the class $\mathscr{S}$ of univalent functions. Later, the Hankel determinant $H_2(n)$ was studied by Hayman \cite{MR0219715} for  mean univalent functions and by Noonam and Thomas \cite{MR0422607}  for mean $p$-valent functions. Noor \cite{MR875965,MR2396299} studied the Hankel determinant for close-to-convex and Bazilevic functions. Sharp estimates were obtained by several authors for the second Hankel determinant,  $H_2(2)=a_2a_4-a_3^2$. Fekete and Szeg\"o \cite{MR1574865} considered the second Hankel determinant $H_2(1)=a_3-a_2^2$ for the class $\mathscr{S}$. They estimated the upper bound for the \emph{Fekete-Szeg\"o functional} $|a_3-\mu a_2^2|$ where $\mu$ is any complex number. This functional plays a very important role in geometeric function theory. For instance, $a_3-a_2^2=S_f(0)/6$ where $S_f$ denote the Schwarzian derivative of a function $f\in\mathscr{S}$. There is a huge literature  on the Fekete-Szeg\"o functional and Hankel determinats for different  subclasses of $\mathscr{S}$. Many authors started investigating the Fekete-Szeg\"o functional and Hankel determinants of  the coefficients of the inverse function of functions belonging to various subclasses of $\mathscr{S}$. Ali \cite{MR2055766} maximized the Fekete-Szeg\"o functional  involving the coefficients of the inverse function of strongly starlike functions of order $\alpha$ ($0<\alpha\leq 1$). Thomas and Verma \cite{MR3646796} considered the class of strongly convex functions of order $\alpha$ and obtained the upper bounds on second Hankel determinant and Fekete-Szeg\"o functional using the coefficients of the inverse function. However, very few papers discussed the third Hankel determinant $H_3(1)=a_3(a_2a_4-a_3^2)-a_4(a_4-a_2a_3)+a_5(a_3-a_2^2)$. Babalola \cite{MR1234560}   investigated the upper bound on $H_3(1)$ for the well-known  classes of bounded turning,  starlike and convex functions while Prajapat \textit{et al.\@} \cite{MR3458966} investigated the same  for a class  of close-to-convex functions. Recently, Raza and Malik \cite{MR3339521} obtained the third Hankel determinant for the class $\mathscr{S}_L^*$ and  Zhang  \textit{et al.\@} \cite{zhang2018third} for the class $\mathscr{S}^*_e$.

The radius of starlikeness of a class $\mathscr{A}$ is the largest number $\mathscr{R}_{\mathscr{S}^*}(\mathscr{A})$ such that the function $f\in\mathscr{A}$ is starlike in the disk, $|z|<\mathscr{R}_{\mathscr{S}^*}$. For instance, the radius of starlikeness for the class $\mathscr{S}$ is $\tanh(\pi/4)\approxeq0.65579$  obtained by Grunsky.

Motivated by the above said works,  in the following section, we  investigate the first two sharp coefficients of the inverse function of functions belonging to  the Ma-Minda class $\mathscr{S}^*(\varphi)$.  We also estimate  the Fekete-Szeg\"o functional for the class $\mathscr{S}^*(\varphi)$.  In particular, we also obtain first five initial inverse coefficient estimates for the two subclasses, $\mathscr{S}^*_e$ and $\SC$. Moreover,  Section \ref{Hankel} provides the upper bound for the second Hankel determinant using the coefficients of the inverse function of functions belonging to  the  class $\mathscr{S}^*(\varphi)$. The bound on the third Hankel determinant for the class $\SC$ is also computed.  In the last section, the sharp $\SC$-radius and $\mathscr{S}^*_e$-radius for several  subclasses of the class $\mathscr{S}^*(\varphi)$ are obtained.

\section{Inverse Coefficient Estimates}\label{ce}
Suppose the function $f(z)=z+\sum_{n=2}^\infty a_nz^n\in \mathscr{S}^*(\varphi)$. Let  $p(z):= zf'(z)/f(z)=1+b_1 z+b_2 z^2+\cdots$ and  $\varphi(z)=1+B_1z+B_2z^2+\cdots$ (where $B_1>0$) be a function with the positive real part in $\mathbb{D}$. Then a simple calculation shows that
\begin{equation*}
a_2  =b_1, \qquad a_3 =\frac{1}{2}(b_1^2+b_2), \qquad a_4  =\frac{1}{6}(b_1^3+3 b_1b_2+2b_3)\end{equation*}and
\begin{equation*}
a_5= \frac{1}{24}(b_1^4+6b_1^2b_2+3b_2^2+8b_1b_3+6b_4).
\end{equation*}
We now express the coefficients $a_n$ ($n=2,3,4,5$) of the function $f\in\mathscr{S}^*(\varphi)$ in terms of the coefficient of the function $\varphi$.  For $f\in\mathscr{S}^*(\varphi)$, we have (see \cite{MR1343506})
\begin{align}
a_2&=\frac{1}{2}B_1c_1,\label{f5} \\
a_3&= \frac{1}{8}\big((B_1^2-B_1+B_2) c_1^2+2 B_1c_2\big),\label{f6} \\
a_4&=\frac{1}{48}\big((B_1^3-3B_1^2+3B_1B_2+2B_1-4B_2+2B_3)c_1^3+2(3B_1^2-4B_1+4B_2)c_1c_2 \notag\\&\quad +8B_1c_3\big) \label{f8},
\intertext{and}a_5 &=\frac{1}{384}\big((B_1^4-6B_1^3+6B_1^2B_2+11B_1^2-22B_1B_2+3B_2^2+8B_1B_3-6B_1+18B_2 \notag  \\ & \quad -18B_3+6B_4)c_1^3   + 4(3B_1^3-11B_1^2+11B_1B_2+9B_1-18B_2+9B_3)c_1^2c_2\label{f9}  \\ & \quad +   12(B_1^2-2B_1+2B_2)c_2^2 + 16(2B_1^2-  3B_1+3B_2) c_1c_3+48B_1c_4\big)\notag.
 \end{align}
Using the power series expansions of the functions $f$ and $f^{-1}$ given by \eqref{inv} in the relation $f(f^{-1}(\omega))=\omega$, or \[\omega=f^{-1}(\omega)+a_2(f^{-1}(\omega))^2+a_3 (f^{-1}(\omega))^3 +\cdots\]we obtain the following  relations for the coefficients of $f^{-1}$
\begin{align*}
A_2&= -a_2,\\
A_3&= 2 a_2^2-a_3,\\
A_4&=-5a_2^3+5 a_2 a_3-a_4\intertext{and}
A_5&=14 a_2^4-21a_2^2a_3+6a_2a_4+3a_3^2-a_5.
\end{align*}
Substituting the values of $a_i$'s from equations \eqref{f5} to \eqref{f9} in the above expresssion, we have
\begin{align}
A_2& = -\frac{1}{2}B_1c_1, \label{A2}\\
A_3& = \frac{1}{8}\big((3B_1^2+B_1-B_2)c_1^2-2B_1c_2\big),\label{A3}\\
A_4& = \frac{1}{24}\big((-8B_1^3-6B_1^2-B_1+2B_2+6B_1B_2-B_3)c_1^3+4(3B_1^2+B_1-B_2)c_1c_2\notag\\&\quad-4B_1c_3\big)\label{A4}
\intertext{and}
A_5&=\frac{1}{384}\big((125B_1^4+150B_1^3+55B_1^2+6B_1+15B_2^2-18B_2-150B_1^2B_2-110B_1B_2\notag\\
&\quad +18B_3+40B_1B_3-6B_4)c_1^4+4(-75B_1^3-55B_1^2-9B_1+18B_2+55B_1B_2\label{A5}\\& \quad -9B_3)c_1^2c_2+12(5B_1^2+2B_1-2B_2)c_2^2+16(10B_1^2+3B_1-3B_2)c_1c_3-48B_1c_4\big)\notag.
\end{align}
Since  \[ |A_3-\mu A_2^2|= |-a_3+2a_2^2-\mu a_2^2|=|a_3-(2-\mu)a_2^2|,  \] the following result follows from  the estimate of the Fekete-Szeg\"o functional   of  functions in $\mathscr{S}^*(\varphi)$.
\begin{theorem}\label{thm1}
	Let the function $f \in \mathscr{S}^*(\varphi)$ and  $f^{-1}(\omega)=\omega+\sum_{n=2}^{\infty}A_n\omega^n$ in some neighbourhood of the origin.
Then
	\begin{equation*}
	|A_3-\mu A_2^2|\leq\frac{B_1}{2}\max\{1, |\nu-1| \}
	\end{equation*}
	where \[ \nu= \frac{1}{B_1}((3-2\mu)B_1^2+B_1-B_2). \]
	In particular, we have
	\begin{equation*}
	  |A_2|\leq B_1\quad \text{and}\quad |A_3|\leq \frac{B_1}{2}\max\left\{1,\frac{1}{B_1}|3B_1^2-B_2| \right\}.
	\end{equation*} The bounds obtained are sharp.
\end{theorem}

Theorem \ref{thm1} can be obtained as a direct application of  the following lemma.
\begin{lemma} \cite{MR1343506}\label{l1}
	Let the function $p(z)=1+\sum_{n=1}^{\infty}p_nz^n\in\mathscr{P}$, then for any complex number $\nu$, we have \[\left|p_2-\frac{\nu}{2} p_1^2\right|\leq 2 \max\{1, |\nu-1|\} =\begin{cases}
	2, \quad & \text{if \ } 0\leq\nu\leq2 \\
	2|\nu-1|,\quad & \text{elsewhere}.
	\end{cases}\]
\end{lemma}

Indeed, using equations \eqref{A2} and \eqref{A3}, we have
	\begin{equation*}
	|A_3-\mu A_2^2|=\frac{B_1}{4}\left|c_2-\frac{1}{2B_1}((3-2\mu)B_1^2+B_1-B_2)c_1^2 \right|.
	\end{equation*} Now the result follows from Lemma \ref{l1}.
	Note that the equality holds for the bound on $A_2$ if and only if the function $f$ is given by $zf'(z)/f(z)=\varphi(\epsilon z)$ where $|\epsilon|=1$.	If $|3B_1^2-B_2|<B_1$, then equality holds for the bound on $A_3$ if and only if the function $f$ is given by $zf'(z)/f(z)=\varphi(\epsilon z^2)$ and if $|3B_1^2-B_2|>B_1$, then equality holds  if and only if the function $f$ is given by $zf'(z)/f(z)=\varphi(\epsilon z)$.  If $|3B_1^2-B_2|=B_1$, then equality holds if and only if the $f$ is given by $zf'(z)/f(z)=\big(\lambda\varphi(\epsilon z)+(1-\lambda)(\varphi(\epsilon z))^{-1}\big)^{-1}$ where $0\leq\lambda\leq1$.	

\begin{remark}
	Letting $\varphi(z)=((1+z)/(1-z))^{\alpha}$ for $0<\alpha\leq1$, then Theorem \ref{thm1}  reduces to \cite[Theorem 1, p.~68]{MR2055766} for the uppper bounds on $|A_2|$ and $|A_3|$. Taking $\varphi(z)=(1+(1-2\alpha)z)/(1-z)$ where $0<\alpha<1$, then Theorem \ref{thm1} simplifies to \cite[Theorem 1, p.~105]{MR689590}. 	If $\varphi(z)=\sqrt{1+z}$, then  $|A_2|\leq1/2$ and $|A_3|\leq 7/16$   which are obtained in \cite[Theorem 4.1, p.~90]{MR3796439} and when $f\in\mathscr{S}^*[A,B]$ where $-1\leq B\leq 1<A$, then $|A_2|\leq A-B$ and $|A_3|\leq (3 A^2 - 5 A B + 2 B^2)/2$ which are proved in \cite[Theorem 2.1, p.~3541]{MR3670865}.
\end{remark}
Next, we estimate the first five initial coefficients of the inverse function of functions belonging to  the functions belonging to $\mathscr{S}^*_e$.
\begin{theorem}\label{thm2}
	Let the function $f\in\mathscr{S}^*_e$ and $f^{-1}(\omega)=\omega+\sum_{n=2}^{\infty}A_n\omega^n$ in some neighbourhood of the origin. Then the following sharp estimates hold:
	\begin{equation*}
	|A_2|\leq 1,\quad |A_3|\leq 5/4, \quad|A_4|\leq 31/18\quad\text{and}\quad |A_5|\leq 361/144.
	\end{equation*}
\end{theorem}
In order to prove the above result, we will using the following lemma in which the inequality \eqref{eq1} was given by \cite{graham2003geometric} while the inequalities \eqref{eq2}   and \eqref{eq3} were given  by \cite{MR652447}.
\begin{lemma}\cite{MR652447,graham2003geometric}\label{lem4}
	If the function $p(z)=1+\sum_{n=1}^{\infty}p_nz^n\in\mathscr{P}$, then \begin{gather}\label{eq1}
	|p_2-\frac{1}{2}p_1^2|\leq 2-\frac{1}{2}|p_1|^2,\\
	\label{eq2}|p_1^3-2 p_1p_2+p_3 |\leq 2
	\intertext{and}\label{eq3}|p_1^4+p_2^2-3p_1^2p_2+2p_1p_3-p_4|\leq 2.
	\end{gather}
\end{lemma}

\begin{proof}[Proof of Theorem \ref{thm2}]
	Since the function $f\in\mathscr{S}^*_e$, we have $B_1=1$, $B_2=1/2$, $B_3=1/6$ and $B_4=1/24$.
	Using the values of $B_1$ and $B_2$ in Theorem \ref{thm1}, the bounds for $A_2$ and $A_3$ are obtained.
 To obtain the upper bound on the fourth inverse coefficient $A_4$,	we use the values of $B_i$'s in equation \eqref{A4} to obtain
	\begin{align*}
	|A_4|&=\frac{1}{24}\left|4c_3-14c_1c_2+\frac{67}{6}c_1^3 \right| \\ &=\frac{1}{24} \left|4(c_3-2c_1c_2+c_1^3)-6 c_1c_2+\frac{43}{6}c_1^3 \right|\\
&	\leq \frac{1}{24}\left[4|c_3-2c_1c_2+c_1^3|+6|c_1|\left|c_2-\frac{43}{36}c_1^2\right| \right]
	\end{align*} where the last step follows from the triangle inequality. Using Lemma  \ref{lem4} and Lemma \ref{l1} with $\nu=43/18$, we have
\begin{align*}
|A_4|\leq \frac{1}{24}\left[8+12\cdot 2\left|\frac{43}{18}-1\right| \right]=\frac{31}{18}.
\end{align*} From equation \eqref{A5}, we get
\begin{align}
|A_5|& =\frac{1}{2304}\left|1261c_1^4-2496c_1^2c_2+432c_2^2+1104c_1c_3-288c_4\right|\notag\\
&= \frac{1}{8}\left|\frac{1261}{288}c_1^4-\frac{26}{3}c_1^2c_2+\frac{3}{2}c_2^2+\frac{23}{6}c_1c_3-c_4 \right|\notag\\
&=\frac{1}{8}\left|B+\frac{11}{6}c_1 C +\frac{5}{4}c_1^2D-\frac{1}{2}DE+\frac{13}{288}c_1^4 \right|\label{eq4}
\end{align}
where \begin{align*}
B&=c_1^4+c_2^2-3c_1^2c_2+2c_1c_3-c_4,\\
C&=c_1^3-2c_1c_2+c_3,\\
D&=c_1^2-c_2
\intertext{and}
 E&=c_2-\frac{1}{2}c_1^2.
\end{align*}
Lemma \ref{l1} and Lemma \ref{lem4} readily show that \[|B|\leq 2, \quad |C|\leq2, \quad |D|\leq2 \quad\text{and}\quad |E|\leq2-\frac{1}{2}|c_1|^2.  \]
Using triangle inequality in \eqref{eq4} along with the above inequalities and the fact that $|c_1|\leq2$, we obtain
\begin{align*}
|A_5|& \leq \frac{1}{8}\left[|B|+\frac{11}{6}|c_1||C|+\frac{5}{4}|c_1|^2|D|+\frac{1}{2}|D||E|+\frac{13}{288}|c_1|^4  \right]\\
& \leq \frac{1}{8}\left[4+\frac{11}{3}|c_1|+2|c_1|^2+\frac{13}{288}|c_1|^4 \right]\leq\frac{361}{144}.
\end{align*}
Let the function $f_0\colon\mathbb{D}\to\mathbb{C}$ be defined by \[f_0(z)=z\exp\left(\int_0^z\frac{e^{\epsilon t}-1}{t}dt\right)=z+\epsilon z^2+\frac{3\epsilon^2}{4}z^3+\frac{17\epsilon^3}{36}z^4+\frac{19\epsilon^4}{72}z^5+\cdots \] where $|\epsilon|=1$. Then $f_0(0)=f'_0(0)-1=0$ and $zf'_0(z)/f_0(z)=e^{\epsilon z}$. Therefore the function $f_0\in\mathscr{S}^*_e$ and we have
\[f_0^{-1}(\omega)=\omega-\epsilon \omega^2+\frac{5\epsilon^2}{4}\omega^3-\frac{31\epsilon^3}{18}\omega^4+\frac{361\epsilon^4}{144}\omega^5+\cdots. \]
 Hence all the  bounds estimated above are sharp for the function $f_0$.
\end{proof}
In the last result of this section, we determine the first five initial coefficients of the inverse function of the functions belonging to the subclass $\SC$.
\begin{theorem}\label{thm3}
	Let the function $f\in\SC$ and $f^{-1}(\omega)=\omega+\sum_{n=2}^{\infty}A_n\omega^n$ for all $\omega$ in some neighbourhood of the origin. Then
	\begin{equation*}
	|A_n|\leq \frac{\sqrt{2}-1}{n-1}\qquad (\text{for }n=2,3,4) \quad\text{and}\quad |A_5|\leq \frac{69}{\sqrt{2}}-\frac{387}{8}.
	\end{equation*}
	First three estimated bounds  are sharp.
\end{theorem}
\begin{proof}
Since the function $f\in\SC$, we get $B_1=1/k$, $B_2=2/k^2$, $B_3=2/k^3$ and $B_4=2/k^4$ where $k=\sqrt{2}+1$.  Therefore  Theorem \ref{thm1} yield the desired upper bounds for $A_2$ and $A_3$.
Now substituting the values of $B_i$'s in equation \eqref{A4}, we obtain
	\begin{align*}
	|A_4|&=\frac{1}{24 k^3}\left|4k^2c_3-4k(1+k) c_1c_2+(-2+2k+k^2)c_1^3 \right| \\ &=\frac{1}{6k} \left|c_3-\frac{1}{k}(1+k) c_1c_2+\left(\frac{-2+2k+k^2}{4k^2}\right)c_1^3 \right|.\end{align*}
	Then $\beta=(1+k)/2k$ and $\alpha=(-2+2k+k^2)/4k^2$ satisfy the hypothesis of \cite[Lemma 3, p.~66]{MR2055766}. Consequently, we have the desired upper bound for $A_4$. From equation \eqref{A5}, we have
	\begin{align*}
	|A_5|& =\frac{1}{8k}\left|\left(\frac{-47-34k+19k^2+6k^3}{48k^3}\right)c_1^4-\left(\frac{-17+19k+9k^2}{12k^2}\right)c_1^2c_2+ \left(\frac{1+2k}{3k}\right)c_2^2\right.\\
	& \quad +\left.\left(\frac{4+3k}{4k}\right)c_1c_3-c_4\right|\\
	&=\frac{1}{8k}\left|B+\left(\frac{4-3k}{3k}\right)c_1 C +\left(\frac{-34-35k+12k^2}{24k^2}\right)c_1^2D-\left(\frac{1-2k}{4k}\right)DE \right.\\
	& \quad+\left.\left(\frac{-47+34k+19k^2-6k^3}{48k^3}\right)c_1^4 \right|
	\end{align*}where $B$, $C$, $D$ and $E$ are same as in proof of Theorem \ref{thm2}. Applying  similar technique as in  Theorem \ref{thm2}, we obtain
	\begin{align*}
	|A_5|& \leq \frac{1}{8k}\left[|B|+\left(\frac{-4+3k}{3k}\right)|c_1| |C| +\left(\frac{34+35k-12k^2}{24k^2}\right)|c_1|^2|D|+\left(\frac{-1+2k}{4k}\right)|D||E| \right.\\
	& \quad+\left.\left(\frac{-47+34k+19k^2-6k^3}{48k^3}\right)|c_1|^4 \right]\\
	& \leq \frac{1}{8k}\left[4-\frac{1}{k}+\frac{2(-4+3k)}{3k}|c_1|+\left(\frac{17+19k-9k^2}{6k^2}\right)|c_1|^2\right.\\& \quad\left. +\left(\frac{-47+34k+19k^2-6k^3}{48k^3}\right)|c_1|^4 \right]\\& \leq\frac{-47+68k+38k^2}{24k^4}=\frac{69}{\sqrt{2}}-\frac{387}{8}.
	\end{align*}
	Define the functions $f_i\colon\mathbb{D}\to\mathbb{C}$ ($i=1,2,3$) by \begin{gather}f_1(z)=z\exp\left(\int_0^z\frac{\varphi_\mathcal{R}(\epsilon t)-1}{t}dt\right)=z+\frac{\epsilon}{k} z^2+\frac{3\epsilon^2}{2k^2}z^3+\frac{11\epsilon^3}{6k^3}z^4+\frac{53\epsilon^4}{24k^4}z^5+\cdots,\notag
	\\
	f_2(z)=z\exp\left(\int_0^z\frac{\varphi_\mathcal{R}(\epsilon t^2)-1}{t}dt\right)=z+\frac{\epsilon^2}{2k}z^3+\frac{5\epsilon^4}{8k^2}z^5+\cdots\label{ff2}
	\intertext{and}
	f_3(z)=z\exp\left(\int_0^z\frac{\varphi_\mathcal{R}(\epsilon t^3)-1}{t}dt\right)=z+\frac{\epsilon^3}{3k}z^4+\frac{\epsilon^6}{18k^2}z^7+\cdots\notag \end{gather}where $|\epsilon|=1$ and $\varphi_{\mathcal{R}}$ is given by \eqref{rf}. Clearly,  the functions $f_i\in\SC$ for $i=1,2,3$ and we have
	\begin{gather*}f_1^{-1}(\omega)=\omega-\frac{\epsilon}{k} \omega^2+\frac{\epsilon^2}{2k}\omega^3+\frac{2\epsilon^3}{3k^3}\omega^4-\frac{47\epsilon^4}{24 k^4}+\cdots,
	\\
	f_2^{-1}(\omega)=\omega-\frac{\epsilon^2}{2k}\omega^3+\frac{\epsilon^4}{8k^2}\omega^5+\cdots
	\intertext{and}
	f_3^{-1}(\omega)=\omega-\frac{\epsilon^3}{3k}\omega^4+\cdots. \end{gather*}
	Hence the bounds of $A_2$, $A_3$ and $A_4$ are sharp for the functions $f_1$, $f_2$ and $f_3$ respectively.
\end{proof}
\section{Hankel Determinants}\label{Hankel}
We begin this section by determining the upper bound on the second Hankel determinant involving the coefficients of the inverse function of the function in the class $\mathscr{S}^*(\varphi)$.
\begin{theorem}\label{Thm1}
	Let the function $f\in\mathscr{S}^*(\varphi)$ and $f^{-1}(\omega)=\omega+\sum_{n=2}^{\infty}A_n\omega^n$ for all $\omega$ in some neighbourhood of the origin.
	\begin{enumerate}[1.]
		\item If $B_1$, $B_2$ and $B_3$ satisfy the conditions \[|3B_1^2-B_2|\leq B_1,\qquad\left|5B_1^4-6B_1^2B_2-3B_2^2+4B_1B_3 \right|-3B_1^2\leq0 \]
		then \[ |A_2A_4-A_3^2|\leq \frac{B_1^2}{4}.\]
		\item If $B_1$, $B_2$ and $B_3$ satisfy the conditions \[|3B_1^2-B_2|\geq B_1,\qquad\left|5B_1^4-6B_1^2B_2-3B_2^2+4B_1B_3 \right|-B_1\left|3B_1^2-B_2\right|-2B_1^2\geq0  \]
		or the conditions
		\[|3B_1^2-B_2|\leq B_1,\qquad\left|5B_1^4-6B_1^2B_2-3B_2^2+4B_1B_3 \right|-3B_1^2\geq0 \]
		then\[ |A_2A_4-A_3^2|\leq \frac{1}{12}\left|5B_1^4-6B_1^2B_2-3B_2^2+4B_1B_3 \right|.\]
		\item If $B_1$, $B_2$ and $B_3$ satisfy the conditions \[|3B_1^2-B_2|> B_1,\qquad\left|5B_1^4-6B_1^2B_2-3B_2^2+4B_1B_3 \right|-B_1\left|3B_1^2-B_2\right|-2B_1^2\leq0  \]
		then \[ |A_2A_4-A_3^2|\leq \frac{B_1^2}{12}\left(\frac{\splitfrac{3\left|5B_1^4-6B_1^2B_2-3B_2^2+4B_1B_3 \right|}	{-4B_1\left|3B_1^2-B_2\right|-(3B_1^2-B_2)^2-4B_1^2}}{\left|5B_1^4-6B_1^2B_2-3B_2^2+4B_1B_3 \right|	-2B_1\left|3B_1^2-B_2\right|-B_1^2}\right).\]
	\end{enumerate}
		
\end{theorem}
To prove the above result, we shall use the following lemma which is given by Libera and Z{\l}otkiewicz \cite{MR681830}.
\begin{lemma}\cite{MR681830}\label{l2}
	Let the function $p\in\mathscr{P}$ and $p(z)=1+\sum_{n=1}^{\infty}p_nz^n$, then
	\begin{align*}
	2 p_2& = p_1^2+\gamma(4-p^2_1)\\
	4 p_3&=p_1^3+2p_1(4-p_1^2)\gamma-p_1(4-p_1^2)\gamma^2+2(4-p_1^2)(1-|\gamma|^2)z
	\end{align*} for some complex valued $\gamma$ and $z$ satisfying $|\gamma|\leq 1$ and $|z|\leq 1$.
\end{lemma}
\begin{proof}[Proof of Theorem \ref{Thm1}]
	Using expressions \eqref{A2}, \eqref{A3} and \eqref{A4}, we have
	\begin{align*}A_2A_4-A_3^2&= \frac{B_1}{192} \bigg[\left(5B_1^3-\frac{3 B_2^2}{B_1}+4B_3+6B_1^2-2B_2+B_1-6B_1B_2 \right)c_1^4 \\ & \quad+ 4(-3B_1^2-B_1+B_2)c_1^2c_2-12B_1c_2^2 +16B_1c_1c_3 \bigg]. \end{align*}
	Let us suppose that
	\begin{align}
	d_1 & =16B_1, \quad d_2	 =4(-3B_1^2-B_1+B_2), \notag \\
	d_3 & =-12 B_1, \quad d_4  =5B_1^3-\frac{3 B_2^2}{B_1}+4B_3+6B_1^2-2B_2+B_1-6B_1B_2,\label{d}\\
	T& = \frac{B_1}{192}.\notag
	\end{align}
	This gives
	\begin{equation}
	|A_2A_4-A_3^2|=T|d_1c_1c_3	+d_2c_1^2c_2+d_3c_2^2	+d_4c_1^4|.
	\end{equation}
	Since $\mathscr{S}^*(\varphi)$ is rotationally invariant and if the function  $p\in\mathscr{P}$, then $p(e^{i\theta}z) \in\mathscr{P}$ (where $\theta$ is a real), we can always suppose that $c_1>0$ and since $|c_1|\leq 2$, without loss of generality assume that $c_1=c\in[0,2]$.
In view of Lemma \ref{l2}, we have
	\begin{align*}|A_2A_4-A_3^2| & = \frac{T}{4} \big|(d_1+2d_2+d_3+4d_4)c^4+2c^2(4-c^2)(d_1+d_2+d_3)\gamma\\ & \quad+(4-c^2)\gamma^2(-d_1c^2+d_3(4-c^2)) +2d_1c(4-c^2)(1-|\gamma|^2)z \big).
	\end{align*}
	Applying the triangle inequality in above equation,  replacing $|\gamma|$ by $\mu$ and substituting the values of $d_1$, $d_2$, $d_3$ and $d_4$ from \eqref{d}, we have
	\begin{align*}|A_2A_4-A_3^2|& \leq \frac{T}{4} \bigg[4c^4\left|5B_1^3-\frac{3 B_2^2}{B_1}+4B_3-6B_1B_2 \right| +8c^2(4-c^2) |3B_1^2-B_2|\mu \\ & \quad+ 4B_1(4-c^2)(12+c^2)\mu^2+32B_1c(4-c^2)(1-\mu^2) \bigg]\\
	& = T \bigg[c^4\left|5B_1^3-\frac{3 B_2^2}{B_1}+4B_3-6B_1B_2 \right| +2c^2(4-c^2) |3B_1^2-B_2|\mu \\ & \quad+ 8B_1c(4-c^2)+B_1(4-c^2)(c-2)(c-6)\mu^2 \bigg]=:F(c,\mu).
   \end{align*}
  	For fixed $c$, since $\partial F/\partial \mu>0$ in the region  $\Omega=\{ (c,\mu)\colon0\leq c\leq2, 0\leq\mu\leq1\}$, $F(c,\mu)$ is an increasing function of $\mu$ in the closed interval $[0,1]$ which implies  the function $F(c,\mu)$ attains its maximum value at $\mu=1$ for some fixed $c\in[0,2]$, that is,\[ \max F(c,\mu)=F(c,1)=:G(c) \] where \begin{align*}G(c)&= \frac{B_1}{192} \bigg[c^4\left|5B_1^3-\frac{3 B_2^2}{B_1}+4B_3-6B_1B_2 \right| +2c^2(4-c^2) |3B_1^2-B_2| \\ & \quad+ B_1(4-c^2)(12+c^2) \bigg]\\
  	&= \frac{B_1}{192} \bigg[c^4\left(\left|5B_1^3-\frac{3 B_2^2}{B_1}+4B_3-6B_1B_2 \right|-2|3B_1^2-B_2|-B_1 \right)\\ & \quad +8c^2\left(|3B_1^2-B_2|-B_1\right) +48B_1 \bigg].
  	\end{align*}
	Let us set
	\begin{align}
	P & =\left|5B_1^3-\frac{3 B_2^2}{B_1}+4B_3-6B_1B_2 \right|-2|3B_1^2-B_2|-B_1,\notag\\
	Q & = 8\left(|3B_1^2-B_2|-B_1\right),\label{pqr}\\
	R &=48B_1.\notag
	\end{align}
Since
\begin{equation*}
\max_{0\leq t	\leq4}(Pt^2+Qt+R)=\begin{cases}
R, \quad & Q\leq0, P\leq-Q/4\\
16P+4Q+R, \quad & Q\geq0, P\geq-Q/8 \text{ or } Q\leq0, P\geq -Q/4\\
(4PR-Q^2)/4P, \quad & Q>0, P\leq-Q/8
\end{cases}
\end{equation*}
we have
\begin{equation*}
|A_2A_4-A_3^2|\leq\frac{B_1}{192}\begin{cases}
R, \quad & Q\leq0, P\leq-Q/4\\
16P+4Q+R, \quad & Q\geq0, P\geq-Q/8 \text{ or } Q\leq0, P\geq -Q/4\\
(4PR-Q^2)/4P, \quad & Q>0, P\leq-Q/8
\end{cases}
\end{equation*}
where $P$, $Q$ are $R$ are same is in \eqref{pqr}. This completes the proof.
\end{proof}
As a consequence of Theorem \ref{Thm1}, we have the following:
\begin{corollary}
	\begin{enumerate}[1.]
	\item If the function $f\in\mathscr{S}^*$, then $|A_2A_4-A_3^2|\leq 3$.
	\item If the function $f\in\mathscr{S}^*_L$, then $|A_2A_4-A_3^2|\leq 19/280$.
	\item If the function $f\in\mathscr{S}^*_e$, then $|A_2A_4-A_3^2|\leq 29/98$.
	\item If the function $f\in\SC$, then $|A_2A_4-A_3^2|\leq 1/4k^2$.
	\end{enumerate}
\end{corollary}

Now let the function $f\in\SC$, then $B_1=1/k$, $B_2=2/k$, $B_3=2/k^2$. Using the values of $B_i$'s in  \eqref{f5}, \eqref{f6} and \eqref{f8},  we get
\begin{align}
a_2&=\frac{1}{2k}c_1 \label{f2}\\
a_3&= \frac{1}{8k^2}(2kc_2+(3-k)c_1^2) \label{f3}\\
a_4&=\frac{1}{48k^3}\left((11-11k+2k^2)c_1^3+2(11-4k)kc_1c_2+8k^2c_3\right). \label{f4}
\end{align}
The following estimate of the Fekete-Szeg\"o functional for $f\in\SC$ is a direct consequence of  Lemma \ref{l1}:
 \begin{equation}
|a_3-\mu a_2^2|\leq \frac{1}{2k} \max\left\{1, \frac{1}{k} |2\mu-3|\right\}, \qquad k=\sqrt{2}+1\label{cor1} \end{equation}
 and in particular, we have \[|a_2|\leq \frac{1}{k} \quad \text{and}\quad |a_3|\leq \frac{3}{2k^2}. \]
The extremal function  $h$ for the class $\SC$ is given by
\begin{align}
h(z) & := \frac{k^2 z}{(k-z)^2}e^{-z/k} \notag \\
& = z+\frac{1}{k}z^2 +\frac{3}{2 k^2}z^3+\frac{11}{6k^3}z^4+
\cdots.
\end{align}
Hence we conclude the following:
\begin{conjecture}\label{conj1}
Let $f\in\SC$ and $f(z)=z+\sum_{n=2}^{\infty}a_nz^n$, then \[|a_n|\leq \frac{1}{k^{n-1}}\left(\sum_{p=0}^{n-1}(-1)^p \frac{n-p}{p!}\right).\]
\end{conjecture}
Following result can be easily obtained by putting the values of $B_i$'s mentioned above in \cite[Theorem 1. p.~3]{MR3073988}:
\begin{theorem}\label{th3}
	Let $f\in\SC$ and $f(z)=z+\sum_{n=2}^{\infty}a_nz^n$, then \[|a_2a_4-a_3^2|\leq \frac{1}{4k^2}\approx 0.0428932. \] The bound obtained is sharp for the function $f_2$ defined by \eqref{ff2}.
\end{theorem}
\begin{theorem}\label{th2}
	Let $f\in\SC$ and $f(z)=z+\sum_{n=2}^{\infty}a_nz^n$, then \[|a_2a_3-a_4|\leq \frac{5220 + 3683 \sqrt{2} + 359 \sqrt{359 + 246 \sqrt{2}} + 	246 \sqrt{718 + 492 \sqrt{2}}}{1458(1+\sqrt{2})^5}\approx 0.244395.\]
\end{theorem}
\begin{proof}
	Using expressions \eqref{f2}, \eqref{f3} and \eqref{f4}, we have
	\[a_2a_3-a_4=- \frac{1}{24k^3} \left((1-4k+k^2)c_1^3+4k(2-k)c_1c_2+4k^2c_3 \right).   \]
	Letting $c_1=c\in[0,2]$ and using Lemma \ref{l2}, we have
	\begin{equation*}a_2a_3-a_4 = -\frac{1}{24k^3} \big(c^3+4k(4-c^2)xc-k^2(4-c^2)x^2c  +2k^2(4-c^2)(1-|x|^2)z \big).   \end{equation*}
	With the help of  same technique as used in previous theorem, an application of triangle inequality and the fact that $1-|x|^2\leq 1$ give
	\begin{equation*}|a_2a_3-a_4|\leq \frac{1}{24k^3} \big( c^3+4k(4-c^2)\mu c +k^2(4-c^2)\mu^2 c+2k^2(4-c^2)\big)=:F(c,\mu).   \end{equation*}
	Since $\partial F/\partial \mu>0$ for any fixed $c\in[0,2]$ and for all $\mu\in[0,1]$, we can say that $F(c,\mu)$ is an increasing function of $\mu$ and hence \[ \max F(c,\mu)=F(c,1)=:G(c) \] where \[G(c)= \frac{1}{24k^3} \big(c^3+4k(4-c^2) c +k^2(4-c^2)c+2k^2(4-c^2)\big).\]
	Since \[G''\left(\frac{2(-k^{2}+M)}{3N}  \right)  =-0.31492<0\] where $\zeta=(-12k+45k^2+24k^3+4k^4)^{1/2}$ and $\eta=-1+4k+k^2$, the maximum value of $G$ occurs at $c=2(-k^{2}+\zeta)/(3\eta)$.
	Therefore \begin{align*} |a_2a_3-a_4|&\leq G\left(\frac{2(-k^{2}+\zeta)}{3\eta}  \right)\\
	&=\frac{144 k^4 + 16 k^5 - 24 \zeta + 12 k^2 (-15 + 4 \zeta) + k^3 (243 + 8 \zeta) + 	9 k (3 + 10 \zeta)}{81 k^{2} \eta^2} \\
	&=\frac{5220 + 3683 \sqrt{2} + 359 \sqrt{359 + 246 \sqrt{2}} +
	246 \sqrt{718 + 492 \sqrt{2}}}{1458(1+\sqrt{2})^5}. \qedhere \end{align*}
\end{proof}
Using  Equation \eqref{cor1},  Conjecture \ref{conj1}, Theorem \ref{th3} and Theorem \ref{th2}, we obtain the following bound on the third Hankel determinant for the functions in the class $\SC$:
\begin{conjecture}
	Let the function $f\in\SC$,  then \begin{align*}|H_3(1)|& \leq \frac{4293 + 1458 k + \frac{88 (144 k^4 + 16 k^5 - 24 \zeta + 12 k^2 (-15 + 4 \zeta) + k^3 (243 + 8 \zeta) +
		9 k (3 + 10 \zeta))}{\eta^2}}{3888 k^5} \\ &\approx 0.0563448\end{align*}
where $\zeta=(-12k+45k^2+24k^3+4k^4)^{1/2}$ and $\eta=-1+4k+k^2$.
\end{conjecture}

\section{Radius Problems}\label{rad}
Raina and Sok\'o{\l} \cite{MR3419845} discussed the subclass $\mathscr{S}^*_q:=\mathscr{S}^*(z+\sqrt{1+z^2})$ associated with lune and Sharma \textit{et al.\@}  \cite{MR3536076} investigated the class $\mathscr{S}^*_C:= \mathscr{S}^*(1+4z/3+2z^2/3 )$ associated with cardiod. For $\alpha\in(0,1)$, Kargar \textit{et al.\@} \cite{MR3933532} (see also \cite{MR3804996}) studied the class  $\mathscr{BS}^*(\alpha):= \mathscr{S}^*(G_{\alpha})$ where $G_{\alpha}(z):=1+z/(1-\alpha z^2))$. Let $\mathscr{CS}^*(\alpha)$ be the class of close-to-star functions of type $\alpha$ which is defined by
\[\mathscr{CS}^*(\alpha)=\left\{f\in\mathscr{A}\colon \frac{f}{g}\in\mathscr{P},\; g\in\mathscr{S}^*(\alpha)  \right\}. \]
 Sok\'o{\l} and Stankiewicz \cite{MR1473947} estimated the radius of convexity for functions in the class $\mathscr{S}^*_L$.  Recently, Kumar and Ravichandran \cite{MR3496681} and Mendiratta \textit{et al.\@} \cite{MR3394060}  estimated the sharp $\SC$-radii and $\mathscr{S}^*_e$-radii, respectively for various well-known classes of functions. For example, they estimated the radius of convexity, $\SC$-radius and $\mathscr{S}^*_e$-radius for the class $\mathscr{S}^*[A,B]$,  $\mathcal{W}:=\{f\in\mathcal{A}\colon \real (f(z)/z) >0, z\in\mathbb{D}\}$, $\mathscr{F}_1:=\{f\in\mathscr{A}\colon f/g\in\mathscr{P} \text{ for some }g\in\mathcal{W} \}$, $\mathscr{F}_2:=\{f\in\mathscr{A}\colon |f(z)/g(z)-1|<1 \text{ for some }g\in\mathcal{W} \}$ and so forth. In this section, we  compute the sharp $\SC$-radius and $\mathscr{S}^*_e$-radius for various other well-known subclasses.
\begin{theorem}\label{thmD}
	The $\SC$-radii for the subclasses $\mathscr{CS}^*(\alpha)$, $\mathscr{S}^*_q$ and $\mathscr{BS}^*(\alpha)$, $\mathscr{M}(\beta)$-radius for the class $\SC$ and $\mathscr{S}^*_L$-radius for the class $\SC$ are given by:
\begin{enumerate}[(a)]
		\item $\mathscr{R}_{\SC} (\mathscr{CS}^*(\alpha))=\rho_0 :=(2-\alpha+\sqrt{7-6\alpha+\alpha^2})/(-3+2\alpha )$
		\item $\mathscr{R}_{\SC}(\mathscr{S}^*_q)=\big(-2+\sqrt{2}+\sqrt{-4+4\sqrt{2}}\big)/2\approx 0.350701$ which is the smallest positive root of the equation $4r^4-4r^2+(57-40\sqrt{2})=0$
		\item $\mathscr{R}_{\SC}(\mathscr{BS}^*(\alpha))=\big(-(3+2\sqrt{2})+\sqrt{4\alpha+17+12\sqrt{2}}\big)/2\alpha$
		\item $\mathscr{R}_{\mathscr{M}(\beta)} (\SC)=\begin{cases}
		1 \quad & \text{if } \beta \geq 2\\
		k(-\beta +\sqrt{\beta^2+4\beta-4}) /2\quad  &\text{if }\beta \leq 2
		\end{cases} $
		\item $\mathscr{R}_{\mathscr{S}^*_L}(\SC)=(-1 + \sqrt{2}) (-4 - 3 \sqrt{2} + \sqrt{62 + 44 \sqrt{2}})/2 \approx 0.601232$
\end{enumerate}respectively. The radii obtained are sharp.
\end{theorem}
The subclass of $\mathscr{P}$ which satisfies $\real p(z)>\alpha$ where $0\leq\alpha<1$ is denoted by $\mathscr{P}(\alpha)$. In general, for $|B|\leq 1$ and $A\not=B$, the class $\mathscr{P}[A,B]$ consists of all those functions $p$ with the normalization $p(0)=1$ satisfying $p(z)\prec (1+Az)/(1+Bz)$.   The following lemmas  will be used in our investigation:

\begin{lemma}\cite{MR0313493}\label{lem1}
	If $p\in \mathscr{P}(\alpha)$, then \[\left|\frac{zp'(z)}{p(z)} \right| \leq \frac{2r(1-\alpha)}{(1-r)(1+(1-2\alpha)r)}, \quad |z|=r<1. \]
\end{lemma}
\begin{lemma}\cite{MR1624955}\label{lem2}
	If $p\in\mathscr{P}[A,B]$, then \[\left|p(z)-\frac{1-ABr^2}{1-B^2r^2} \right|\leq \frac{|A-B|r}{1-B^2r^2},  \quad |z|=r<1.  \]
\end{lemma}

\begin{lemma}\cite{MR3496681}\label{rr}
	For $2(\sqrt{2}-1)<a<2$, let $r_a$ be defined by \[r_a=
	\begin{cases}
	a-2(\sqrt{2}-1),  &\text{if}\quad  2(\sqrt{2}-1)<a\leq \sqrt{2}; \\
	2-a, &\text{if} \quad \sqrt{2}\leq a<2.
	\end{cases} \]Then $\{w \in\mathbb{C} \colon |w-a|<r_a \} \subset
	\varphi_\mathcal{R}(\mathbb{D}) $  where \[\varphi_\mathcal{R}(\mathbb{D}):=\{w\in\mathbb{C}\colon | w+ (w^2+4w-4)^{1/2}|<2/k \}. \]
\end{lemma}

\begin{proof}[Proof of Theorem \ref{thmD}](a)  Let the function $f\in\mathscr{CS}^*(\alpha)$ and the function $g\in\mathscr{S}^*(\alpha)$ be such that $p(z)=f(z)/g(z)\in\mathscr{P}$. Then $zg'(z)/g(z)\in\mathscr{P}(\alpha)$ and  Lemma \ref{lem2} gives \begin{equation*}
	\left|\frac{zg'(z)}{g(z)}-\frac{1+(1-2\alpha)r^2}{1-r^2} \right| \leq \frac{2(1-\alpha)r}{1-r^2}.
	\end{equation*}
	Since the function $p\in\mathscr{P}$, Lemma \ref{lem1} yields\begin{equation*}
	\left| \frac{zp'(z)}{p(z)}\right| \leq \frac{2r}{1-r^2}.
	\end{equation*} Using the above estimates in the identity \[\frac{zf'(z)}{f(z)}=\frac{zg'(z)}{g(z)}+\frac{zp'(z)}{p(z)} \] we can see that \begin{equation}\label{e13}
	\left|\frac{zf'(z)}{f(z)} - \frac{1+(1-2\alpha)r^2}{1-r^2}\right|\leq \frac{2(2-\alpha)r}{1-r^2}.
	\end{equation}
	Let $0\leq r\leq \rho_0$. Then it can be easily seen that if $a:=(1+(1-2\alpha)r^2)/(1-r^2)$, then $a \leq 2$.  Therefore from Lemma \ref{rr}, we can see that the disk \eqref{e13} lies inside the domain $\varphi_\mathcal{R}(\mathbb{D})$ if and only if
	\[ \frac{2(2-\alpha)r}{1-r^2} \leq2- \frac{1+(1-2\alpha)r^2}{1-r^2}. \]The last inequality reduces to $-1 + 2 (2 - \alpha) r + (3 - 2 \alpha) r^2\leq 0$, which holds if  $r\leq \rho_0$ and the result follows. Consider the functions $f$, $g\in\mathscr{A}$ defined by \[f(z)=\frac{z(1+z)}{(1-z)^{3-2\alpha}}\quad\text{and}\quad g(z)=\frac{z}{(1-z)^{2-2\alpha}}. \] Since \[\real \frac{f(z)}{g(z)} = \real \frac{1+z}{1-z}>0 \quad \text{and}\quad \real \frac{zg'(z)}{g(z)} = \real \frac{1+(1-2\alpha)z^2}{1-z^2}>\alpha  \] then $g\in\mathscr{S}^*(\alpha)$ and hence $f\in\mathscr{CS}^*(\alpha)$.  Also at the point $z=\rho_0$, we see that \begin{align*}\frac{zf'(z)}{f(z)} & =\frac{1+2(2-\alpha)z+(1-2\alpha)z^2}{1-z^2}\\  & = \frac{1+(1-2\alpha)\rho_0^2}{1-\rho_0^2}- \frac{2(2-\alpha)\rho_0}{1-\rho_0^2}= 2.\end{align*}
	This proves the sharpness of the result.
	
 \noindent (b) Since the function $f$ is in the class $\mathscr{S}^*_q$, for $|z|=r$, we have \begin{align}\left|\frac{zf'(z)}{f(z)}-1 \right| & \leq|z+\sqrt{1+z^2}-1|  \leq 1-r-\sqrt{1-r^2}.\label{e2} \end{align} In view of Lemma \ref{rr}, the disk \eqref{e2} lies in the domain $\varphi_\mathcal{R}(\mathbb{D})$ if $1-r-\sqrt{1-r^2} \leq 3-2\sqrt{2}$ or $4r^4-4r^2+(57-40\sqrt{2})\leq 0$ which gives the desired radius estimate and this estimate is best possible for the function \begin{equation}\label{e3} f_q(z):=z\exp(q(z)-\log (1-z+q(z))+\log 2-1 ). \end{equation}

\noindent (c) Let the function $f\in\mathscr{BS}^*(\alpha)$ and $|z|=r$.   Then a simple calculation yields
\begin{align}\left|\frac{zf'(z)}{f(z)}-1 \right|  &\leq\left|\frac{z}{1-\alpha z^2}\right|  \leq \frac{r}{1-\alpha r^2}.\label{e4} \end{align}
Using Lemma \ref{rr}, we see  that the disk \eqref{e4} is contained in the domain $\varphi_\mathcal{R}(\mathbb{D}$) if \[ \frac{r}{1-\alpha r^2}\leq  3-2\sqrt{2} \] or $\alpha r^2+( 3+2\sqrt{2})r-1\leq 0$. This gives the required radius estimate. The function defined by \begin{equation}\label{e7}
	 f_{B}(z):=z\exp \left(\frac{\tanh^{-1}(\sqrt{\alpha}z)}{\sqrt{\alpha}} \right)
	 \end{equation} proves that the estimation is sharp.

\noindent (d) Let the function $f \in\SC$.
	
\noindent  \textit{Case} 1. Let $\beta\geq 2$. For $|z|=r<1$, using the definition of subordination,  it is easy to see that \[\real \frac{zf'(z)}{f(z)} \leq \max_{|z|=r}\varphi_{\mathcal{R}}(z) \leq 1+\frac{r}{k}\left(\frac{k+r}{k-r}\right)<\frac{k^2+1}{k(k-1)}\leq\beta. \]
	 \textit{Case} 2. Let $\beta\leq  2$. For $|z|=r<k(-\beta +\sqrt{\beta^2+4\beta-4}) /2$, using the same technique as in Case 1, it follows that \[ \real \frac{zf'(z)}{f(z)} \leq 1+\frac{r}{k}\left(\frac{k+r}{k-r}\right)<\beta. \]
	 This proves the desired result.  Sharpness follows   for the function
\begin{equation}\label{eqR}
f_r(z)=\frac{k z}{(k-z)^2} e^{-z/k}.
\end{equation}
	
	 \noindent (e) Since $f \in\SC$, we have
	 \[|\varphi_\mathcal{R}(z)-1|^2 = \frac{r^2}{k^2}\left( \frac{k^2 + r^2 + 2 k r \cos t}{k^2 + r^2 - 2 k r \cos t} \right)  < (\sqrt{2}-1)^2\] if \[\left|\frac{zf'(z)}{f(z)}-1\right| <\sqrt{2}-1,\quad |z|=r< \frac{1}{2} (-1 + \sqrt{2}) (-4 - 3 \sqrt{2} + \sqrt{62 + 44 \sqrt{2}}).  \] Therefore the result follows from \cite[Lemma 2.2, p.~6559]{MR2879136}.  The radius estimate is sharp for the function $f_r$ defined by \eqref{eqR}.
\end{proof}
Next result yields the sharp radius estimates related to the class  $\mathscr{S}^*_e$.
\begin{theorem}\label{thmE}
 	The $\mathscr{S}^*_e$-radii for the subclasses $\mathscr{S}^*_L$, $\mathscr{S}^*_q$, $\SC$,  $\mathscr{S}^*_C$ and $\mathscr{BS}^*(\alpha)$ are given as:
	\begin{enumerate}[(a)]
		\item $\mathscr{R}_{\mathscr{S}^*_e}(\mathscr{S}^*_L)=(e^2-1)/e^2 \approx 0.864665$
		\item $\mathscr{R}_{\mathscr{S}^*_e}(\mathscr{S}^*_q)= (-2e+\sqrt{-4e^2+8e^4})/(4e^2) \approx 0.498824$ which is the smallest positive root of the equation $4r^4-4r^2+((e^2-1)/e^2)^2=0$
		\item $\mathscr{R}_{\mathscr{S}^*_e}(\SC)= (k - 2 e k +k \sqrt{1 - 8 e + 8 e^2} )/(2 e)\approx 0.780444 $
		\item $\mathscr{R}_{\mathscr{S}^*_e}(\mathscr{S}^*_C)=(-2e+\sqrt{10e^2-4e})/2e \approx 0.395772$
		\item $\mathscr{R}_{\mathscr{S}^*_e}(\mathscr{BS}^*(\alpha))=(-e+\sqrt{e^2+4(e-1)^2\alpha})/(2\alpha (e-1))$
	\end{enumerate}respectively. The results are all sharp.
\end{theorem}
To prove our estimations, we will make use of the following result.

\begin{lemma}\cite{MR3394060}\label{re}
	For $1/e<a<e$, let $r_a$ be defined by \[r_a=
	\begin{cases}
	a-e^{-1}, & \text{if} \quad e^{-1}<a\leq (e+e^{-1})/2; \\
	e-a,  &\text{if}\quad  (e+e^{-1})/2\leq a<e.
	\end{cases} \]Then $\{w \in\mathbb{C} \colon |w-a|<r_a \} \subset \{w\in\mathbb{C}\colon |\log w|<1 \}$.
\end{lemma}
\begin{proof}[Proof of Theorem \ref{thmE}]
 (a) Since the function $f\in\mathscr{S}^*_L$, we have
	\begin{equation*}
	\left|\frac{zf'(z)}{f(z)}-1 \right| =|\sqrt{1+z}-1|\leq 1-\sqrt{1-r}, \qquad |z|\leq r.
	\end{equation*}
	By applying Lemma \ref{re}, we note that the function $f\in\mathscr{S}^*_e$ if $1-\sqrt{1-r}\leq 1-1/e$ which leads to the inequality $r\leq e^2-1/e^2$. The obtained radius estimate is sharp for the function $f_L$ defined as \[\frac{zf'_L(z)}{f'_L(z)}=\sqrt{1+z}. \]
	(b) In view of Lemma \ref{re}, we see that the disk \eqref{e2} is contained in the domain $e^z(\mathbb{D}):=\{w\in\mathbb{C}\colon|\log w|<1 \}$ provided $1-r-\sqrt{1-r^2}\leq 1-1/e$ or $r+\sqrt{1-r^2}\geq 1/e$ or equivalently \[4r^4-4r^2+\left(\frac{e^2-1}{e^2} \right)^2\leq 0.\] The previous inequality yields the desired estimate for the radius. The sharpness follows for the function $f_q$ defined by \eqref{e3}.
	
\noindent(c) Let the function $f\in\SC$. Then  \begin{equation}\label{e5}
\left|\frac{zf'(z)}{f(z)}-1 \right|\leq \frac{r}{k}\left(\frac{k+r}{k-r} \right), \qquad |z|\leq r.
\end{equation}
Using Lemma \ref{re}, we see that the disk \eqref{e5} lies in the domain $\{w\in\mathbb{C}\colon|\log w|<1 \}$ if \[ \frac{r}{k}\left(\frac{k+r}{k-r} \right) \leq1-\frac{1}{e}. \]The above inequality simplies to \[er^2+k(2e-1)r-k^2 (e-1)\leq 0 \] which gives the required radius estimate.  Sharpness follows for the function $f_r$ defined by \eqref{eqR}.

(d) Let the function $f\in\mathscr{S}^*_C$. Then a simple calculation gives \begin{equation}\label{e6}
\left|\frac{zf'(z)}{f(z)}-1 \right|\leq \frac{1}{3}(4r+2r^2), \qquad |z|\leq r.
\end{equation}
Using Lemma \ref{re}, we note that the disk \eqref{e6} is contained in the domain $\{w\in\mathbb{C}\colon|\log w|<1 \}$ if $ (4r+2r^2)/3\leq1-1/e$ or $2er^2+4er-3(e-1)\leq0$. The last inequality gives \[r\leq  \frac{-2e+\sqrt{10e^2-6e}}{2 e}. \] The result is sharp for the function \[f_C(z):=z\exp\left(\frac{4z}{3}+\frac{2z^2}{3} \right).\]
(e) Using Lemma \ref{re}, note that the disk \eqref{e4} lies in the domain $e^z(\mathbb{D})$ provided\[\frac{r}{1-\alpha r^2}\leq 1-\frac{1}{e}. \] By a  simple computation, the last inquality becomes $\alpha(e-1)r^2+er-(e-1)\leq0$ which gives \[r\leq \frac{-e+\sqrt{e^2+4(e-1)^2\alpha}}{2\alpha(e-1)}. \] The function $f_B$ defined by \eqref{e7} shows the sharpness of this radius estimate.
\end{proof}


\begin{thebibliography}{10}
	
	\bibitem{MR3436767}
M. F. Ali\ and\ A. Vasudevarao, Coefficient inequalities and Yamashita's conjecture for some classes of analytic functions, J. Aust. Math. Soc. {\bf 100} (2016), no.~1, 1--20.
	
	\bibitem{MR2055766}
R. M. Ali, Coefficients of the inverse of strongly starlike functions, Bull. Malays. Math. Sci. Soc. (2) {\bf 26} (2003), no.~1, 63--71.
	
	\bibitem{MR2879136}
R. M. Ali, N. K. Jain\ and\ V. Ravichandran, Radii of starlikeness associated with the lemniscate of Bernoulli and the left-half plane, Appl. Math. Comput. {\bf 218} (2012), no.~11, 6557--6565.
	
\bibitem{MR1234560}
K. O. Babalola, On ${H}_3(1)$ {H}ankel determinant for some classes of univalent functions, Ineq. Thoery and Appl. {\bf 6} (2010),	pp.~1--7.
	
	
\bibitem{MR3804996}
N. E. Cho, S. Kumar, V. Kumar\ and \ V. Ravichandran,  Differential subordination and radius estimates for starlike functions associated with the	{B}ooth lemniscate, Turk. J. Math. {\bf42} (2018), pp.~1380--1399.
	
	
	\bibitem{MR1574865}
M. Fekete\ and\ G. Szeg\"{o}, Eine Bemerkung Uber Ungerade Schlichte Funktionen, J. London Math. Soc. {\bf 8} (1933), no.~2, 85--89.
	
	\bibitem{graham2003geometric}
 I. Graham, {\it Geometric function theory in one and higher dimensions}, 	CRC Press, 2003.
	
	\bibitem{MR0219715}
W. K. Hayman, On the second Hankel determinant of mean univalent functions, Proc. London Math. Soc. (3) {\bf 18} (1968), 77--94.
	
	\bibitem{MR0328059}
W. Janowski, Some extremal problems for certain families of analytic functions. I, Ann. Polon. Math. {\bf 28} (1973), 297--326.
	
	\bibitem{MR2296897}
G. P. Kapoor\ and\ A. K. Mishra, Coefficient estimates for inverses of starlike functions of positive order, J. Math. Anal. Appl. {\bf 329} (2007), no.~2, 922--934.
	
	\bibitem{MR3933532}
R. Kargar, A. Ebadian\ and\ J. Sok\'{o}\l, On Booth lemniscate and starlike functions, Anal. Math. Phys. {\bf 9} (2019), no.~1, 143--154.
	
	\bibitem{MR689590}
J. G. Krzy\.{z}, R. J. Libera\ and\ E. Z\l otkiewicz, Coefficients of inverses of regular starlike functions, Ann. Univ. Mariae Curie-Sk\l odowska Sect. A {\bf 33} (1979), 103--110 (1981).

	\bibitem{MR3496681}
S. Kumar\ and\ V. Ravichandran, A subclass of starlike functions associated with a rational function, Southeast Asian Bull. Math. {\bf 40} (2016), no.~2, 199--212.


	
	\bibitem{MR3073988}
S. K. Lee, V. Ravichandran\ and\ S. Supramaniam, Bounds for the second Hankel determinant of certain univalent functions, J. Inequal. Appl. {\bf 2013}, 2013:281, 17 pp.
	
	\bibitem{MR652447}
R. J. Libera\ and\ E. J. Z\l otkiewicz, Early coefficients of the inverse of a regular convex function, Proc. Amer. Math. Soc. {\bf 85} (1982), no.~2, 225--230.
	
	\bibitem{MR681830}
R. J. Libera\ and\ E. J. Z\l otkiewicz, Coefficient bounds for the inverse of a function with derivative in $P$, Proc. Amer. Math. Soc. {\bf 87} (1983), no.~2, 251--257.
	
	\bibitem{MR813267}
R. J. Libera\ and\ E. J. Z\l otkiewicz, The coefficients of the inverse of an odd convex function, Rocky Mountain J. Math. {\bf 15} (1985), no.~3, 677--683.
	
	\bibitem{MR1140278}
R. J. Libera\ and\ E. J. Z\l otkiewicz, L\"{o}wner's inverse coefficients theorem for starlike functions, Amer. Math. Monthly {\bf 99} (1992), no.~1, 49--50.
	
	\bibitem{MR1512136}
K. L\"{o}wner, Untersuchungen \"{u}ber schlichte konforme Abbildungen des Einheitskreises. I, Math. Ann. {\bf 89} (1923), no.~1-2, 103--121.
	
	\bibitem{MR1343506}
W. C. Ma\ and\ D. Minda, A unified treatment of some special classes of univalent functions, in {\it Proceedings of the Conference on Complex Analysis (Tianjin, 1992)}, 157--169, Conf. Proc. Lecture Notes Anal., I, Int. Press, Cambridge, MA.
	
	\bibitem{MR3394060}
R. Mendiratta, S. Nagpal\ and\ V. Ravichandran, On a subclass of strongly starlike functions associated with exponential function, Bull. Malays. Math. Sci. Soc. {\bf 38} (2015), no.~1, 365--386.
	
	\bibitem{MR0422607}
J. W. Noonan\ and\ D. K. Thomas, On the second Hankel determinant of areally mean $p$-valent functions, Trans. Amer. Math. Soc. {\bf 223} (1976), 337--346.
	
	\bibitem{MR2396299}
K. I. Noor, On certain analytic functions related with strongly close-to-convex functions, Appl. Math. Comput. {\bf 197} (2008), no.~1, 149--157.
	
	\bibitem{MR875965}
K. I. Noor\ and\ S. A. Al-Bany, On Bazilevic functions, Internat. J. Math. Math. Sci. {\bf 10} (1987), no.~1, 79--88.

	\bibitem{MR0185105}
Ch. Pommerenke, On the coefficients and Hankel determinants of univalent functions, J. London Math. Soc. {\bf 41} (1966), 111--122.
	
	\bibitem{MR0215976}
Ch. Pommerenke, On the Hankel determinants of univalent functions, Mathematika {\bf 14} (1967), 108--112.
	
	\bibitem{MR3458966}
J. K. Prajapat, D. Bansal, A. Singh\ and \ A. K. Mishra, Bounds on third Hankel determinant for close-to-convex functions, Acta Univ. Sapientiae Math. {\bf 7} (2015), no.~2, 210--219.
	
	\bibitem{MR3419845}
R. K. Raina\ and\ J. Sok\'{o}\l, Some properties related to a certain class of starlike functions, C. R. Math. Acad. Sci. Paris {\bf 353} (2015), no.~11, 973--978.
	
	\bibitem{MR1624955}
V. Ravichandran, F. R\o nning\ and\ T. N. Shanmugam, Radius of convexity and radius of starlikeness for some classes of analytic functions, Complex Variables Theory Appl. {\bf 33} (1997), no.~1-4, 265--280.
	
	\bibitem{MR3670865}
V. Ravichandran\ and\ S. Verma, Estimates for coefficients of certain analytic functions, Filomat {\bf 31} (2017), no.~11, 3539--3552.
	
	\bibitem{MR3339521}
M. Raza\ and\ S. N. Malik, Upper bound of the third Hankel determinant for a class of analytic functions related with lemniscate of Bernoulli, J. Inequal. Appl. {\bf 2013}, 2013:412, 8 pp.

	\bibitem{MR783568}
M. S. Robertson, Certain classes of starlike functions, Michigan Math. J. {\bf 32} (1985), no.~2, 135--140.
	
	\bibitem{MR0313493}
G. M. Shah, On the univalence of some analytic functions, Pacific J. Math. {\bf 43} (1972), 239--250.
	
\bibitem{MR0994916} T. N. Shanmugam, Convolution and differential subordination, Internat. J. Math. Math. Sci. {\bf 12} (1989), no.~2, 333--340.

	\bibitem{MR3536076}
K. Sharma, N. K. Jain\ and\ V. Ravichandran, Starlike functions associated with a cardioid, Afr. Mat. {\bf 27} (2016), no.~5-6, 923--939.
	
	\bibitem{MR1473947}
J. Sok\'{o}\l\ and\ J. Stankiewicz, Radius of convexity of some subclasses of strongly starlike functions, Zeszyty Nauk. Politech. Rzeszowskiej Mat. No. 19 (1996), 101--105.
	
	\bibitem{MR3796439}
J. Sok\'{o}\l\ and\ D. K. Thomas, Further results on a class of starlike functions related to the Bernoulli lemniscate, Houston J. Math. {\bf 44} (2018), no.~1, 83--95.
	
	\bibitem{MR3646796}
D. K. Thomas\ and\ S. Verma, Invariance of the coefficients of strongly convex functions, Bull. Aust. Math. Soc. {\bf 95} (2017), no.~3, 436--445.

	\bibitem{MR1304483}
B. A. Uralegaddi, M. D. Ganigi\ and\ S. M. Sarangi, Univalent functions with positive coefficients, Tamkang J. Math. {\bf 25} (1994), no.~3, 225--230.
	
	\bibitem{zhang2018third}
 H. Y. Zhang, H. Tang, and X. M. Niu, Third-order hankel determinant
for certain class of analytic functions related with exponential function, 	Symmetry {\bf 10} (2018), p.~501.
	
\end{thebibliography}
\end{document}